\documentclass[11pt]{amsart}
\usepackage{amssymb,mathrsfs,graphicx,enumerate, stmaryrd}
\usepackage{amsmath,amsfonts,amssymb,amscd,amsthm,bbm}
\usepackage[retainorgcmds]{IEEEtrantools}
\usepackage{colortbl}
\usepackage{hyperref}
\usepackage{tikz}
\usetikzlibrary{patterns}
\newcommand{\greatcircle}[5][]{%
\path[#1,pattern=north west lines,pattern color=#1!60,rotate=#5,dashed] (#2) circle [x radius=#3, y radius=#4];
\begin{scope}[rotate=#5]
\clip (#3,0) rectangle ([xshift=-0.1,yshift=-0.1]-#3,-#4);
\draw[#1] (#2) circle [x radius=#3, y radius=#4];
\end{scope}
}

\title[   ]{Emergent behaviors in group ring flocks}

\author[Ha]{Seung-Yeal Ha}
\address[Seung-Yeal Ha]{\newline Department of Mathematical Sciences and Research Institute of Mathematics \newline Seoul National University, Seoul 08826 and \newline
Korea Institute for Advanced Study, Hoegiro 85, 02455, Seoul, Republic of Korea}
\email{syha@snu.ac.kr}

\author[Park]{Hansol Park}
\address[Hansol Park]{\newline Department of Mathematical Sciences\newline Seoul National University, Seoul 08826, Republic of Korea}
\email{hansol960612@snu.ac.kr}

\newtheorem{theorem}{Theorem}[section]
\newtheorem{lemma}{Lemma}[section]
\newtheorem{corollary}{Corollary}[section]
\newtheorem{proposition}{Proposition}[section]
\newtheorem{remark}{Remark}[section]

\newtheorem{definition}{Definition}[section]

\newcommand{\bbr}{\mathbb R}

\newcommand{\bbs}{\mathbb S}
\newcommand{\bbz}{\mathbb Z}

\newcommand{\bbc}{\mathbb C}
\newcommand{\bbk}{\mathbb K}

\begin{document}

\date{\today}


\thanks{\textbf{Acknowledgment.} The work of S.-Y. Ha is supported by National Research Foundation of Korea (NRF-2020R1A2C3A01003881) and the work of H. Park is supported by Basic Science Research Program through the National Research Foundation of Korea(NRF) funded by the Ministry of Education (2019R1I1A1A01059585). Both authors would like to thank Mr. Lee, Gangsan for the suggestion of {\it ``Group ring"}, and both authors are also would like to thank {\it Stack exchange}(link :\url{https://tex.stackexchange.com/questions/161888/how-to-draw-big-circle-intersection-of-a-plane-with-a-sphere}) for Figure 2. }

\begin{abstract}
We present a first-order aggregation model on a group ring, and study its asymptotic dynamics. In a positive coupling strength regime, we show that the flow generated by the proposed model tends to an equilibrium manifold asymptotically. For this, we introduce a Lyapunov functional which is non-increasing along the flow, and using the temporal decay of the nonlinear functional and the LaSalle invariance principle, we show that the flow converges toward an equilibrium manifold asymptotically. We also show that the structure of an equilibrium manifold is strongly dependent on the structure of an underlying group.  
\end{abstract}

\maketitle \centerline{\date}


\section{Introduction} \label{sec:1}
\setcounter{equation}{0}
Collective behaviors of natural and man-made complex systems often appear in our nature and society, e.g., aggregation of bacteria \cite{T-B}, swarming of fish \cite{D-M1, T-T}, flocking of multi-agent system \cite{C-Sm, V-C-B-C-S}, flashing of fireflies \cite{B-B}, heart beatings by pacemaker cells and hand clapping \cite{Pe}, self-organization of robot and multi-agent systems \cite{F-T-H, O2} etc. We refer to survey articles and books \cite{A-B, A-B-F, D-F-M-T, D-B1, H-K-P-Z, P-R, St, VZ, Wi1} for collective dynamics. Despite of their ubiquity, mathematical modelings for collective behaviors are done only a half century ago by Winfree and Kuramoto in \cite{Ku2, Wi2}. Since their seminal works, several phenomenological consensus(or aggregation) models on manifolds were studied in literature. To name a few, the Kuramoto model on the unit circle \cite{B-C-M, C-H-J-K, C-S, D-X, D-B1, H-K-R, Ku2}, quaternion model \cite{D-F-M-T},  Lohe sphere model \cite{C-C-H, C-H5,J-C,Lo-1, Lo-2, M-T-G,T-M, Zhu},  matrix-valued consensus models \cite{B-C-S, D-F-M, De, H-R2, H-Ry}, the Lohe tensor model \cite{H-R3, H-R4} and the Schr\"{o}dinger-Lohe model \cite{C-H4, H-R1}. We also refer to \cite{A-M-D, F-P-P, F-Z, Ma, S-S, T-A-V} on the related consensus models on Riemannian manifolds).  \newline

In this paper, we are interested in the collective modeling of time-varying algebraic objects which are elements of a group ring \cite{P1, P2}.   Throughout the paper, we take the field $\bbk = \bbr$ or $\bbc$. Then, the group ring $\bbk[G]$ is a $\bbk$-vector space over a finite group $G$, and we consider an ensemble of $N$ time-varying elements ${\mathcal X} = \{ x^i(t) \} \subset \bbk[G]$. Then, we are interested in the following  simple question:

\vspace{0.1cm}

\begin{quote}
`` Can we design a first-order aggregation model for ${\mathcal X}$ which exhibits emergent behaviors? "
\end{quote}

\vspace{0.1cm}

Our main goal of this paper is to answer the aforementioned posed question, i.e., we provide a first-order aggregation model and study how the underlying group $G$ affects system dynamics. Before we present our main results, we briefly discuss elementary operators in the group ring $\bbk[G]$. Let $x$ and $y$ be elements in $\bbk[G]$ represented by
\begin{equation} \label{A-0}
 x=\sum_{g\in G} x_g g \quad \mbox{and} \quad y=\sum_{g\in G} y_g g,
 \end{equation}
where coefficients $x_g$ and $y_g$ are elements of $\bbk$. Then, addition, scalar multiplication, group ring multiplication and hermitian conjugate are defined as follows: for $\lambda \in \bbk$ and $x, y \in \bbk[G]$ in \eqref{A-0},
\begin{align}
\begin{aligned} \label{A-1}
& x+y :=\sum_{g\in G} (x_g+y_g)g, \quad \lambda x := \sum_{g\in G} (\lambda x_g) g, \\
& xy :=\sum_{g',g''\in G} x_{g'}y_{g''}(g' *g''), \quad x^\dagger :=\sum_{g\in G}\overline{x_g} g^{-1},
\end{aligned}
\end{align}
where $g' * g''$ is a group operation of $g'$ and $g''$ in $G$ and $g^{-1}$ is an inverse of $g$. Here $\overline{x_g}$ stands for the complex conjugate of $x_g$. \newline

The main results of this paper are three-fold. First, we propose a first-order aggregation model on $\bbk[G]$ for an ensemble ${\mathcal X}$: 
\begin{equation} \label{A-1-2}
\frac{dx^i}{dt} = \frac{\kappa}{N} \sum_{k=1}^{N} \Big[ x^k(x^i)^\dagger x^i-x^i(x^k)^\dagger x^i  \Big], \quad i = 1, \cdots, N, 
\end{equation}
where $\kappa$ is a nonnegative constant. Then, system \eqref{A-1-2} can also be written as a mean-field form:
\begin{equation}\label{A-2}
\frac{dx^i}{dt} = \kappa \Big[ x^c(x^i)^\dagger x^i-x^i(x^c)^\dagger x^i  \Big ], \quad i = 1, \cdots, N, 
\end{equation}
where $\displaystyle x^c :=\frac{1}{N}\sum_{i=1}^Nx^i$ is the centroid of the ensemble ${\mathcal X}$.

Note that the right-hand side of \eqref{A-2} is well-defined in the group ring $\bbk[G]$ via operations \eqref{A-1}. Other than the elementary operations in \eqref{A-1}, we can also introduce a trace, associated inner product and norm on $\bbk[G]$ as well (see Definition \ref{D2.1} and \eqref{New-0}). Then, by straightforward calculations, we can show that $\|x^i\|_F$ is a constant of motion for \eqref{A-2} (see Lemma \ref{L3.1}):
\[ \|x^i(t) \|_F =  \|x^i(0) \|_F, \quad t \geq 0. \]
On the other hand, system \eqref{A-2} can be rewritten in terms of coefficient $x_g^i \in \bbk$ (see Lemma \ref{L2.2}):
\begin{equation*}\label{A-3}
\frac{dx_g^i}{dt}=\sum_{g_1, g_2\in G}\left(x_{g_1}^c\overline{x^i_{g_2}}x^i_{g_2g_1^{-1}g}-x_{g_1}^i\overline{x^c_{g_2}}x^i_{g_2g_1^{-1}g}\right),\quad \forall g\in G,~~i = 1, \cdots, N.
\end{equation*}
Secondly, we show that the flow $X(t)$ converges to the equilibrium manifold ${\mathcal E}(\bbk[G])$:
\[ \mathcal{E}(\bbk[G]) :=\left\{ X := (x^1, \cdots, x^N) \in \bbk[G]^N:~x^c(x^i)^\dagger x^i -x^i(x^c)^\dagger  x^i  =0_G, \quad \forall~i = 1, \cdots, N \right\}, \]
where $0_G = \sum_{g \in G} 0 g$. 
For the convergence, we introduce a Lyapunov functional which is non-increasing along the flow $X(t)$ and its zero set for orbital derivative coincides with the equilibrium manifold ${\mathcal E}(\bbk[G])$. Then, by the LaSalle invariance principle, one can show that for a positive coupling strength $\kappa > 0$ and general initial data, 
\[
\lim_{t \to \infty} \mathrm{dist}(X(t), {\mathcal E}(\bbk[G])) = 0,
\]
where $\displaystyle \mathrm{dist}(X(t), {\mathcal E}(\bbk[G]))=\inf_{X_{eq} \in{\mathcal E}(\bbk[G])}\|X(t)-X_{eq}\|_F$. We refer to Theorem \ref{T3.1} and Section \ref{sec:3.1} for details. \newline

Thirdly, we provide a complete classification on the structure of  the equilibrium manifold ${\mathcal E}(\bbk[G])$ as a disjoint union of three subsets, and present several quantitative estimates regarding the aforementioned three sets (see Theorem \ref{T4.1}). \newline

The rest of this paper is organized as follows. In Section \ref{sec:2}, we briefly review the basic properties of a group ring $\bbk[G]$ and study basic properties of the aggregation model \eqref{A-2} on the group ring $\bbk[G]$. In Section \ref{sec:3}, we present asymptotic convergence of the flow generated by system \eqref{A-2} toward the equilibrium manifold. In Section \ref{sec:4}, we provide structural results for the structure of the equilibrium manifold ${\mathcal E}(\bbk[G])$ for al finite group $G$. Finally, Section \ref{sec:5} is devoted to a brief summary of our main results and some remaining issues for a future work.  

\section{Preliminaries} \label{sec:2}
\setcounter{equation}{0}
In this section, we first study basic properties of the group ring $\bbk[G]$ and then present basic properties for the aggregation model \eqref{A-1} such as a conservation law and invariance by an automorphism group on $G$. For notational simplicity, we also denote $g_1 * g_2$ by $g_1 g_2$ by deleting $*$.  First, we introduce a concept of a trace in following definition.
\begin{definition} \label{D2.1}
Let $\displaystyle x = \sum_{g \in G} x_g g$ be an element in $\bbk[G]$. Then, the trace of $x$ is given by the coefficient of the identity element $e$:
\[   \mathrm{tr} (x )= x_e. \]
\end{definition}
\begin{remark}
It is easy to see that for $x, y \in \bbk[G]$,
\[ \mathrm{tr}(xy) := \sum_{g^{\prime} *g^{\prime \prime} = e}  x_{g'}y_{g''}. \]
\end{remark}
Next, we list basic properties for hermitian conjugate and trace in the following lemma.
\begin{lemma} \label{L2.1}
Let $x$ and $y$ be elements of the group ring $\bbk[G]$. Then we have
\[
(xy)^\dagger=y^\dagger x^\dagger, \quad \mathrm{tr}(xy)= \mathrm{tr}(yx).
\]
\end{lemma}
\begin{proof} We set 
\[  x = \sum_{g\in G}x_g g \quad \mbox{and} \quad y = \sum_{g\in G}y_g g.  \]
\noindent $\bullet$~(First relation):~By group ring multiplication and hermitian conjugation defined in \eqref{A-1}, we obtain the first identity:
\begin{align*}
\begin{aligned}
(xy)^\dagger &=\left(\sum_{g_1, g_2\in G} x_{g_1}y_{g_2} g_1 g_2 \right)^\dagger=\sum_{g_1,g_2\in G}\overline{x_{g_1}}\overline{y_{g_2}}(g_1 g_2)^{-1} \\
&=\sum_{g_1,g_2\in G}\overline{y_{g_2}}\overline{x_{g_1}}g_2^{-1}  g_1^{-1}  = \Big(  \sum_{g_2\in G} \overline{ y_{g_2}} g_2^{-1} \Big) \Big( \sum_{g_1\in G} \overline{x_{g_1}} g_1^{-1}  \Big) =y^\dagger x^\dagger.
\end{aligned}
\end{align*}

\vspace{0.2cm}

\noindent $\bullet$~(Second relation):~We use \eqref{A-1} and the fact that the left and right inverses are equal for all elements in $G$ to see
\[
\mathrm{tr}(xy) =\sum_{g_1*g_2=e}x_{g_1}y_{g_2}=\sum_{g_2 * g_1=e}y_{g_2}x_{g_1}=\mathrm{tr}(yx).
\]
\end{proof}
In what follows, we study several properties of the aggregation model \eqref{A-2} on $\bbk[G]$. First, we provide other alternative representation of system \eqref{A-2} as a component form. 
\begin{lemma} \label{L2.2}
Let $\{ x^i(t)= \sum_{g\in G} x^i_g(t) g \}$ be a solution to system \eqref{A-2}. Then, for each $g \in G$, the coefficient $x^i_g$ satisfies  
\begin{align}\label{B-1}
\frac{dx_g^i}{dt}= \kappa \sum_{g_1, g_2\in G}\left(x_{g_1}^c\overline{x^i_{g_2}}x^i_{g_2g_1^{-1}g}-x_{g_1}^i\overline{x^c_{g_2}}x^i_{g_2g_1^{-1}g}\right),\quad i = 1, \cdots, N.
\end{align}
\end{lemma}
\begin{proof} We substitute the ansatz $x^i=\sum_{g\in G}x^i_g g$ into \eqref{A-2}, and then compare the coefficients of $g$ in the resulting relation to derive the desired system. \newline

\noindent $\bullet$~(L.H.S \eqref{A-2}): Note that 
\begin{equation} \label{B-5}
\frac{dx^i}{dt} =\sum_{g\in G}\left(\frac{dx^i_g}{dt}\right)g.
\end{equation}

\vspace{0.2cm}

\noindent $\bullet$~(R.H.S \eqref{A-2}):  By definition of multiplication and hermitian conjugate, one has
\begin{align}
\begin{aligned} \label{B-6}
x^c(x^i)^\dagger x^i&=\sum_{g_1, g_2, g_3 \in G}x^c_{g_1}\overline{x_{g_2}^i} x^i_{g_3} (g_1g_2^{-1}g_3) =\sum_{g\in G}\left(\sum_{g_1g_2^{-1}g_3=g}x^c_{g_1}\overline{x_{g_2}^i} x^i_{g_3}\right)g
\\
&=\sum_{g\in G}\left(\sum_{g_1, g_2\in G}x_{g_1}^c\overline{x^i_{g_2}}x^i_{g_2g_1^{-1}g}\right)g,
\end{aligned}
\end{align}
where we used the relation $g_1g_2^{-1}g_3=g\quad \Longleftrightarrow\quad g_3=g_2g_1^{-1}g$. \newline

In an exact same manner, we also have 
\begin{equation} \label{B-7}
x^i(x^c)^\dagger x^i  = \sum_{g\in G}\Big(\sum_{g_1, g_2\in G} x_{g_1}^i\overline{x^c_{g_2}}x^i_{g_2g_1^{-1}g} \Big) g.
\end{equation}
In \eqref{A-2}, we collect all the relations \eqref{B-5}, \eqref{B-6} and \eqref{B-7} to obtain
\[
\sum_{g\in G} \left(\frac{dx_g^i}{dt}\right) g= \kappa \sum_{g\in G}\Big[ \sum_{g_1, g_2\in G}\left(x_{g_1}^c\overline{x^i_{g_2}}x^i_{g_2g_1^{-1}g}-x_{g_1}^i\overline{x^c_{g_2}}x^i_{g_2g_1^{-1}g}\right) \Big ] g.
\]
Finally, we use the fact that $G$ is a basis to derive the desired system \eqref{B-1}.
\end{proof}
Now, we are ready to introduce the Frobenius inner product and norm as follows. For $ x, y  \in \bbk[G]$, 
\begin{equation} \label{New-0}
\langle x, y \rangle_F := \mathrm{tr}(x^\dagger y), \quad \|x \|_F := \sqrt{\langle x, x \rangle_F}. 
\end{equation}
Then, it is easy to check that $\langle \cdot, \cdot \rangle_F$ and $\| \cdot \|_F$ satisfies all the axioms for inner product and norm. In particular, one has
\[ \langle x, y \rangle_F = \sum_{g \in G} \overline{x_{g}}y_{g}, \quad \langle x, y \rangle_F =  \overline{\langle y, x \rangle_F}.   \]
Moreover, one can also check 
\begin{equation*} \label{B-8}
\| x + y \|_F \leq \|x \|_F + \| y \|_F \quad \mbox{and} \quad  |\langle x, y \rangle_F| \leq \|x \|_F \cdot \| y \|_F. 
\end{equation*}
Next, we show that the norm $\|x \|_F$ is a constant of motion for \eqref{A-2}. 
\begin{proposition}\label{P2.1}
Let $\{x^i\}$ be the solution to system \eqref{A-2} with initial data satisfying 
\[ \|x^{i}(0) \|_F=1\quad i=1, 2, \cdots, N. \]
Then, one has 
\[  \|x^{i}(t)\|_F=1, \quad t > 0, \quad i = 1, \cdots, N. \]
\end{proposition}
\begin{proof} We use  \eqref{B-1} and \eqref{New-0} to find
\begin{equation} \label{B-9}
\frac{d}{dt}\|x^i\|_F^2 =\frac{d}{dt}\sum_{g\in G} \overline{x_g^i} x_g^i=\sum_{g\in G}\left(\frac{d{x}_g^i}{dt}\right) \overline{x_g^i}+(c.c.),
\end{equation}
where $(c.c.)$ is the complex conjugate of the first term in \eqref{B-9}. \newline

Note that  
\begin{align}
\begin{aligned} \label{B-10}
\sum_{g\in G}\left(\frac{d{x}_g^i}{dt}\right) \bar{x}_g^i &=\sum_{g, g_1, g_2\in G}\left(x_{g_1}^c\overline{x^i_{g_2}}x^i_{g_2g_1^{-1}g}-x_{g_1}^i\overline{x^c_{g_2}}x^i_{g_2g_1^{-1}g}\right)\bar{x}_g^i  \\
& =\sum_{g, g_1, g_2\in G}\left(x_{g_1}^c\overline{x^i_{g_2}}x^i_{g_2g_1^{-1}g}\overline{x_g^i}-x_{g_1}^i\overline{x^c_{g_2}}x^i_{g_2g_1^{-1}g}\overline{x_g^i}\right).
\end{aligned}
\end{align}
Then, we use  \eqref{B-9} and \eqref{B-10} to obtain
\begin{align}
\begin{aligned} \label{B-11}
\frac{d}{dt}\|x^i\|_F^2 &=\sum_{g, g_1, g_2\in G}\left(x_{g_1}^c\overline{x^i_{g_2}}x^i_{g_2g_1^{-1}g}\overline{x_g^i}-x_{g_1}^i\overline{x^c_{g_2}}x^i_{g_2g_1^{-1}g}\overline{x_g^i}\right)+(c.c.)\\
&=\sum_{g, g_1, g_2\in G}\left(x_{g_1}^c\overline{x^i_{g_2}} x^i_{g_2g_1^{-1}g}\overline{x_g^i}-x_{g_1}^i\overline{x^c_{g_2}}x^i_{g_2g_1^{-1}g}\overline{x_g^i}+\overline{x_{g_1}^c} x_{g_2}^i \overline{x^i_{g_2g^{-1}_1g}}x_g^i-\overline{x_{g_1}^i} x_{g_2}^c \overline{x_{g_2g_1^{-1}g}^i} x_g^i\right).
\end{aligned}
\end{align}
Since $g_1, g_2$ and $g$ are dummy variables, one has
\begin{align}
\begin{aligned} \label{B-12}
&\sum_{g, g_1, g_2\in G}\left(x_{g_1}^c\overline{x^i_{g_2}} x^i_{g_2g_1^{-1}g}\overline{x_g^i} -x_{g_1}^i\overline{x^c_{g_2}} x^i_{g_2g_1^{-1}g}\overline{x_g^i} +\overline{x_{g_1}^c}x_{g_2}^i\overline{x_{g_2g^{-1}_1g}^i} x_g^i-\overline{x_{g_1}^i} x_{g_2}^c \overline{x_{g_2g_1^{-1}g}^i} x_g^i\right)\\
& =\sum_{g, g_1, g_2\in G}\left(x_{g_1}^c \overline{x^i_{g_2}} x^i_{g_2g_1^{-1}g}\overline{x_g^i} -x_{g_1}^i \overline{x^c_{g_2}} x^i_{g_2g_1^{-1}g}\overline{x_g^i} +\overline{x_{g_2}^c} x_{g_1}^i\overline{x_{g_1g^{-1}_2g}^i} x_g^i-\overline{x_{g_2}^i} x_{g_1}^c \overline{x_{g_1g_2^{-1}g}^i} x_g^i\right)\\
&=\sum_{g, g_1, g_2\in G}\left(x_{g_1}^c(\overline{x^i_{g_2}} x^i_{g_2g_1^{-1}g}\overline{x_g^i} -\overline{x_{g_2}^i}  \overline{x_{g_1g_2^{-1}g}^i} x_g^i)+\overline{x_{g_2}^c} (x_{g_1}^i\overline{x_{g_1g^{-1}_2g}^i} x_g^i-x_{g_1}^ix^i_{g_2g_1^{-1}g} \overline{x_g^i})\right)\\
&=\sum_{g, g_1, g_2\in G}\left(x_{g_1}^c \overline{x^i}_{g_2}(x^i_{g_2g_1^{-1}g}\overline{x_g^i} -  \overline{x_{g_1g_2^{-1}g}^i} x_g^i)+\overline{x_{g_2}^c} x_{g_1}^i(\overline{x_{g_1g^{-1}_2g}^i} x_g^i-x^i_{g_2g_1^{-1}g}\overline{x_g^i})\right)\\
& =\sum_{g, g_1, g_2\in G}x_{g_1}^c \overline{x^i_{g_2}}\left(x^i_{g_2g_1^{-1}g}\overline{x_g^i} -  \overline{x_{g_1g_2^{-1}g}^i} x_g^i\right)+(c.c.)\\
&=\sum_{g_1, g_2\in G}\left[ x_{g_1}^c \overline{x^i_{g_2}} \sum_{g\in G} \Big (x^i_{g_2g_1^{-1}g}\overline{x_g^i} -  \overline{x_{g_1g_2^{-1}g}^i} x_g^i \Big )\right]+(c.c.),
\end{aligned}
\end{align}
where in the second line, we used index exchange transformation $g_1 \longleftrightarrow g_2$. Now, we use the property of the dummy variables:
\[
\sum_{g\in G} x^i_{g_2g_1^{-1}g}\overline{x_g^i} =\sum_{g\in G}x^i_{g_2g_1^{-1}(g_1g_2^{-1}g)}\overline{x_{g_1g_2^{-1}g}^i}=\sum_{g\in G} x^i_{g}\overline{x^i_{g_1g_2^{-1}g}},
\]
i.e., for all $g_1, g_2\in G$, one has 
\begin{equation} \label{B-13}
\sum_{g\in G} (x^i_{g_2g_1^{-1}g} \overline{x_g^i}-  \overline{x_{g_1g_2^{-1}g}^i} x_g^i)=\sum_{g\in G}( x^i_{g}\overline{x^i_{g_1g_2^{-1}g}}- \overline{x_{g_1g_2^{-1}g}^i} x_g^i)=0.
\end{equation}
Finally, we use \eqref{B-12} and \eqref{B-13} to get the desired estimate.
\end{proof}
As a direct corollary of Proposition \ref{P2.1}, we show that system \eqref{A-2} reduces to the Kuramoto model for identical oscillators as a special case.
\begin{corollary} \label{C2.1}
Let $G = \{e \}$ and $\{x^i \}$ be a solution to system \eqref{A-2} with $\|x^i(0) \|_F = 1$, and we set 
\begin{equation} \label{B-14}
x^i = e^{{\mathrm i} \theta^i}, \quad i = 1, \cdots, N. 
\end{equation}
Then the phase configuration $\{\theta^i \}$ satisfies the Kuramoto model for identical oscillators:
\begin{equation*}
 \frac{d\theta^i}{dt} =\frac{2\kappa}{N}\sum_{k=1}^N\sin(\theta^k-\theta^i), \quad i = 1, \cdots, N.
 \end{equation*}
\end{corollary}
\begin{proof} 
By Proposition \ref{P2.1},
\[ x^i = x^i_e e,  \quad |x^i_e| = 1, \quad i = 1, \cdots, N, \]
and its component $x^i_e$ satisfies 
\begin{align*}
\begin{aligned} \label{B-13-1}
\frac{dx^i_e}{dt} = \frac{\kappa}{N} \sum_{k=1}^{N} \Big(  x_e^k -  x_e^i \overline{x_e^k}  x_e^i  \Big), \quad i = 1, \cdots, N.
\end{aligned}
\end{align*}
Next, we substitute the ansatz \eqref{B-14} into \eqref{A-2} to get 
\[
e^{\mathrm{i}\theta^i} {\mathrm i} {\dot \theta}^i =\frac{\kappa}{N}\sum_{k=1}^N \Big(e^{\mathrm{i}\theta^k}-e^{\mathrm{i}(2\theta^i-\theta^k)} \Big ).
\]
After further simplifications, one derive the desired Kuramoto model. 
\end{proof}
\vspace{0.5cm}

Next, we show the invariance of system \eqref{A-2} under a group automorphism on $G$. More precisely, let $\varphi\in\mathrm{Aut}(G)$ be a group automorphism. Then, we can define its extension $\tilde{\varphi}:\bbk[G] \to \bbk[G]$ by
\begin{equation*} \label{B-14}
\tilde{\varphi}\left(\sum_{g\in G}x_g g\right):=\sum_{g\in G}x_g\varphi(g)=\sum_{g\in G}x_{\varphi^{-1}(g)}g.
\end{equation*}
Then, it is easy to see 
\[  \tilde{\varphi}(g) = \tilde{\varphi}(1 g) = 1 \varphi(g) = \varphi(g), \quad \forall~g \in G, \quad \mbox{i.e.,} \quad   \tilde{\varphi} \Big|_{G} = \varphi. \]
\begin{lemma} \label{L2.3}
Let $x, y\in\bbk[G]$ and $\varphi\in\mathrm{Aut}(G)$. Then, its extension $\tilde{\varphi}$ satisfies 
\[ \tilde{\varphi}(xy)=\tilde{\varphi}(x) \tilde{\varphi}(y) \quad \mbox{and} \quad \tilde{\varphi}(x^{\dagger}) = \tilde{\varphi}(x)^{\dagger}. \]
\end{lemma}
\begin{proof} 
\noindent (i)~We use the relation \eqref{A-1} to obtain
\begin{align*}
\begin{aligned}
\tilde{\varphi}(xy) &=\tilde{\varphi} \Big( \sum_{g',g''\in G} x_{g'}y_{g''}(g' g'') \Big) =  \sum_{g',g''\in G} x_{g'}y_{g''} \varphi(g' g'') = \sum_{g',g''\in G} x_{g'}y_{g''}  \varphi(g^{\prime})  \varphi(g^{\prime \prime})  \\
& = \Big( \sum_{g^{\prime} \in G}x_{g^{\prime}}\varphi(g^{\prime})     \Big)   
\Big(  \sum_{g^{\prime \prime} \in G}x_{g^{\prime \prime}}\varphi(g^{\prime \prime}) \Big) =\tilde{\varphi}(x) \tilde{\varphi}(y).
\end{aligned}
\end{align*}
\vspace{0.2cm}

\noindent (ii)~By definition of hermitian conjugation, we have
\begin{align*}
\begin{aligned}
\tilde{\varphi}(x^{\dagger}) &=  \tilde{\varphi} \Big( \sum_{g \in G} \bar{x_g} g^{-1} \Big) = \sum_{g \in G} \bar{x_g} \varphi(g^{-1}) =   \sum_{g \in G} \overline{x_g} \varphi(g)^{-1} = \Big( \tilde{\varphi} \Big( \sum_{g \in G} x_g g \Big) \Big)^{\dagger} = \tilde{\varphi}(x)^{\dagger}.
\end{aligned}
\end{align*}

\end{proof}
In the following proposition, we show that the aggregation model \eqref{A-2} is invariant under the automorphism $\tilde{\varphi} \in \mbox{Aut}(\bbc[G])$. 
\begin{proposition}
Let ${\mathcal X} := \{ x^i \}$ be a solution to system \eqref{A-2}, and $\varphi \in \mbox{Aut}(G)$. Then $y^i$ defined by 
\[
y^i := \tilde{\varphi}(x^i), \quad i = 1, \cdots, N
\]
satisfies the same system \eqref{A-2}:
\[
\frac{dy^i}{dt} = \kappa \Big[ y^c(y^i)^\dagger y^i-y^i(y^c)^\dagger y^i \Big],
\]
where $y^c := \frac{1}{N} \sum_{i=1}^{N} y^i$. 
\end{proposition}
\begin{proof} Recall $x^i$ satisfies 
\[
\frac{dx^i}{dt} = \kappa \Big [(x^c(x^i)^\dagger x^i-x^i(x^c)^\dagger x^i  \Big ].
\]
Then we use Lemma \ref{L2.3} to see
\begin{align*}
\begin{aligned}
\frac{dy^i}{dt} &=\frac{d}{dt} \tilde{\varphi}(x^i)=\tilde{\varphi}\left(\frac{dx^i}{dt}\right)= \kappa \tilde{\varphi}\left(x^c(x^i)^\dagger x^i-x^i(x^c)^\dagger x^i\right)  \\
&= \kappa \Big[  \tilde{\varphi}(x^c) \tilde{\varphi}(x^i)^{\dagger} \tilde{\varphi}(x^i) -  \tilde{\varphi}(x^i) \tilde{\varphi}(x^c)^{\dagger} \tilde{\varphi}(x^i) \Big ] = \kappa \Big[ y^c(y^i)^\dagger y^i-y^i(y^c)^\dagger y^i \Big ].
\end{aligned}
\end{align*}
\end{proof}
%
%

\section{Large-time behavior of the flow} \label{sec:3}
\setcounter{equation}{0}
In this section, we study asymptotic convergence of \eqref{A-2} toward an equilibrium manifold, and analyze the geometric structure of the equilibrium manifold for the special finite groups $G = \bbz_2$ and $\bbz_3$. \newline

Recall the aggregation model:
\begin{equation}\label{C-0}
\frac{dx^i}{dt} = \kappa \Big[ x^c(x^i)^\dagger x^i-x^i(x^c)^\dagger x^i  \Big], \quad i = 1, \cdots, N.
\end{equation}
Then, we define two sets associated with \eqref{C-0}:
\begin{align}
\begin{aligned} \label{C-1}
\mathcal{E}(\bbk[G]) &:=\left\{ X \in \bbk[G]^N:~x^c(x^i)^\dagger x^i -x^i(x^c)^\dagger  x^i  =0_G, \quad \forall~i = 1, \cdots, N \right\}, \\
\mathcal{M}(\bbk[G]) &:=\left\{ X \in \bbk[G]^N:~x^c(x^i)^\dagger-x^i(x^c)^\dagger=0_G, \quad \forall~i = 1, \cdots, N \right\}.
\end{aligned}
\end{align}
Note that the set $\mathcal{E}(\bbk[G]) $ consists of all equilibria, and the completely aggregated state $X_{eq} = (x, \cdots, x) \in \bbk[G]^N$ lies in the equilibrium set ${\mathcal E}(\bbk[G])$. In this case, since $x^c = x$, 
\[  x^c(x^i)^\dagger x^i -x^i(x^c)^\dagger  x^i  =  x x^{\dagger} x - x x^\dagger x = 0_G. \]
This yields
\[ x^c(x^i)^\dagger x^i -x^i(x^c)^\dagger  x^i = ( x^c(x^i)^\dagger -x^i(x^c)^\dagger) x^i = 0_G, \]
we can easily see 
\begin{equation*} \label{C-1-0}
{\mathcal M}(\bbk[G]) \subset {\mathcal E}(\bbk[G]).
\end{equation*}
The other inclusion is true, thus the above two sets are in fact the same (see Remark \ref{R3.1}).  \newline

\subsection{Asymptotic convergence to ${\mathcal E}(\bbk[G])$} \label{sec:3.1}
In this subsection, we study the asymptotic convergence of the flow generated by system \eqref{A-2} using the LaSalle invariance principle \cite{La}. For this,  we define two nonlinear functionals  measuring the degree of aggregation for the configuration $\mathcal{X} = \{ x^i \}$ with $\|x^i \|_F = 1$:
\begin{equation} \label{C-2}
R^2 := \frac{1}{N^2}\sum_{i, j=1}^N \langle x^i, x^j\rangle_F = \|x^c\|_F^2 \quad \mbox{and} \quad V :=\frac{1}{N^2}\sum_{i,j=1}^N\|x^i-x^j\|_F^2. 
\end{equation}
Then, it is easy to see
\begin{equation} \label{C-3}
0 \leq R \leq 1, \quad V =\frac{1}{N^2}\sum_{i, j=1}^N\Big(2-\langle x^i, x^j\rangle_F -\langle x^j, x^i\rangle_F 
 \Big)=2-2R^2 \geq 0.
\end{equation}
In next lemma, we show that $V$ satisfies a Lyapunov estimate along the flow \eqref{A-2}.
\begin{lemma} \label{L3.1}
Suppose that the coupling strength and initial data satisfy 
\[ \kappa > 0 \quad \mbox{and} \quad \|x^{i}(0) \|_F=1\quad \forall~i=1, 2, \cdots, N, \]
and let $\{x^i\}_{i=1}^N$ be a solution to system \eqref{A-2}. Then, one has 
\begin{equation*} \label{C-3-0}
\frac{dV}{dt} = -\frac{2\kappa}{N}\sum_{i=1}^N\|x^c(x^i)^\dagger-x^i(x^c)^\dagger\|_F^2\leq0.
\end{equation*}
\end{lemma}
\begin{proof} Since 
\[ \frac{dV}{dt} = -2 \frac{dR^2}{dt} = -\frac{2}{N^2} \frac{d}{dt} \sum_{i,j = 1}^{N} \langle x^i, x^j \rangle_F. \]
Hence it suffices to check   
\begin{equation} \label{C-3-1}
 \frac{d}{dt}\sum_{i, j = 1}^{N} \langle x^i, x^j\rangle_F =  N\kappa\sum_{i=1}^N\|x^c(x^i)^\dagger-x^i(x^c)^\dagger\|_F^2.
\end{equation}
{\it Derivation of \eqref{C-3-1}}: We use \eqref{A-2} and Lemma \ref{L2.1} to find 
\begin{align*}
\begin{aligned}
\frac{d}{dt}\langle x^i, x^j\rangle_F &=\kappa\left\langle \frac{dx^i}{dt}, x^j \right\rangle_F +\kappa\left\langle x^i, \frac{dx^j}{dt} \right\rangle_F \\
&= \kappa\Big \langle x^c (x^i)^\dagger x^i-x^i(x^c)^\dagger x^i, x^j \Big \rangle_F + \kappa\Big \langle x^i, x^c(x^j)^\dagger x^j-x^j(x^c)^\dagger x^j \Big \rangle_F \\
&=\kappa\mathrm{tr}\left((x^i)^\dagger x^i (x^c)^\dagger x^j-(x^i)^\dagger x^c (x^i)^\dagger x^j+(x^i)^\dagger x^c(x^j)^\dagger x^j-(x^i)^\dagger x^j(x^c)^\dagger x^j\right)\\
&=\frac{\kappa}{N}\sum_{k=1}^N \mathrm{tr} \left((x^i)^\dagger x^i (x^k)^\dagger x^j-(x^i)^\dagger x^k (x^i)^\dagger x^j+(x^i)^\dagger x^k(x^j)^\dagger x^j-(x^i)^\dagger x^j(x^k)^\dagger x^j\right).
\end{aligned}
\end{align*}
We sum up the above relation over $i,j$ to derive \eqref{C-3-1}:
\begin{align*}
\begin{aligned}
&\frac{d}{dt}\sum_{i, j =1}^{N} \langle x^i, x^j\rangle_F \\
&\hspace{1cm} =\frac{\kappa}{N}\sum_{i, j, k=1}^N \mathrm{tr}\left((x^i)^\dagger x^i (x^k)^\dagger x^j-(x^i)^\dagger x^k (x^i)^\dagger x^j+(x^i)^\dagger x^k(x^j)^\dagger x^j-(x^i)^\dagger x^j(x^k)^\dagger x^j\right)\\
&\hspace{1cm}=N\kappa\sum_{i=1}^N \mathrm{tr} \left((x^i)^\dagger x^i (x^c)^\dagger x^c-(x^i)^\dagger x^c (x^i)^\dagger x^c+(x^c)^\dagger x^c (x^i)^\dagger x^i-(x^c)^\dagger x^i(x^c)^\dagger x^i\right) \\
&\hspace{1cm}=N\kappa\sum_{i=1}^N \mathrm{tr} \left( x^i (x^c)^\dagger x^c(x^i)^\dagger-x^c (x^i)^\dagger x^c(x^i)^\dagger +x^c (x^i)^\dagger x^i(x^c)^\dagger - x^i(x^c)^\dagger x^i(x^c)^\dagger\right) \\
&\hspace{1cm}=N\kappa\sum_{i=1}^N\left\langle x^c(x^i)^\dagger-x^i(x^c)^\dagger, x^c(x^i)^\dagger-x^i(x^c)^\dagger\right\rangle_F \\
&\hspace{1cm}=N\kappa\sum_{i=1}^N\|x^c(x^i)^\dagger-x^i(x^c)^\dagger\|_F^2.
\end{aligned}
\end{align*}
This and \eqref{C-2} yield  
\begin{equation} \label{C-3-1-1}
\frac{dR^2}{dt} = \frac{\kappa}{N}\sum_{i=1}^N\|x^c(x^i)^\dagger-x^i(x^c)^\dagger\|_F^2\geq0. 
\end{equation}
Again, \eqref{C-3-1-1} and $\eqref{C-3}_2$ imply the desired estimate.  \newline
\end{proof}
\begin{remark} \label{R3.1} As a ramification of Lemma \ref{L3.1}, we obtain the following assertions:
\begin{enumerate}
\item
Since $V$ is nonnegative and non-increasing over time, there exists an asymptotic limit $V_\infty$ such that 
\[ \lim_{t  \to \infty} V(t) = V_\infty. \]
\item
Suppose that ${\mathcal Y} = \{y^i\}_{i=1}^N\in \mathcal{E}(\bbk[G])$ is an equilibrium, i.e., it is stationary. Thus $R$ must be constant and $\frac{dR}{dt}  = 0$. Then, it follows from Lemma \ref{L3.1} that 
\[
0 = \frac{dR^2}{dt} =\frac{\kappa}{N}\sum_{i=1}^N \|y^c(y^i)^\dagger-y^i(y^c)^\dagger\|_F^2,\quad \mbox{i.e.,} \quad {\mathcal Y} \in {\mathcal M}(\bbk[G]). 
\]

Thus, two sets defined in \eqref{C-1} are the same:
\[
{\mathcal E}(\bbk[G]) =  {\mathcal M}(\bbk[G]).
\] 
\end{enumerate}
\end{remark}

\vspace{0.5cm}

Next, we state our first main result on the asymptotic convergence of the flow.
\begin{theorem}\label{T3.1}
Let $\{x^i\}_{i=1}^N$ be a solution to system \eqref{A-2} with the initial data satisfying 
\[
\|x^{i}(0) \|_F=1\quad \forall i=1, 2, \cdots, N.
\]
Then,  the flow approaches to the equilibrium manifold ${\mathcal E}(\bbk[G])$ asymptotically in the sense that 
\[  \lim_{t \to \infty} \mathrm{dist}(X(t), {\mathcal E}(\bbk[G])) = 0. \]
\end{theorem}
\begin{proof} 
Let $X(t)$ be a given flow. Then, the bounded Lyapunov functional $V$ satisfies Lyapunov estimate (see Lemma \ref{L3.1}):
\[ \frac{dV}{dt} = -\frac{2\kappa}{N}\sum_{i=1}^N\|x^c(x^i)^\dagger-x^i(x^c)^\dagger\|_F^2\leq0. \]
Then, the zero set of the orbital derivative $\frac{dV}{dt}$ exactly coincides with ${\mathcal M}(\bbk[G])$ which is a subset of equilibrium manifold, (in fact, we will see that they are the same). 
Thus, the largest invariant subset of ${\mathcal M}(\bbk[G])$ is itself, and by LaSalle's invariance principle, we can confirm the asymptotic convergence of the flow $X(t)$ to the set ${\mathcal M}(\bbk[G]) = {\mathcal E}(\bbk[G])$. 
\end{proof}
In the following two subsections, we consider the group ring $\bbr[G]$ over the real number field, and study the structure of the equilibrium set for the following two special cases for $G$:
\[ G = \bbz_2\quad\mbox{or}\quad \bbz_3. \]

\subsection{Geometric structure of ${\mathcal E}(\bbr[\bbz_2])$}   \label{sec:3.2}  In this subsection, we show that ${\mathcal E}(\bbr[\bbz_2])$ can be the whole state space itself: every solution is an equilibrium.
\begin{proposition}\label{P3.1}
The equilibrium set ${\mathcal E}(\bbr[\bbz_2])$ is the whole state space:
\[ {\mathcal E}(\bbr[\bbz_2]) = \bbr[\bbz_2]^N. \]
\end{proposition}
\begin{proof} 
We set $G = \bbz_2 = \{ e, a \}$ and recall the equilibrium manifold for \eqref{A-2}:
\begin{equation*} \label{C-3-3}
\mathcal{E}(\bbr[\bbz_2]) :=\left\{ X \in \bbr[\bbz_2]^N:~x^c(x^i)^\dagger-x^i(x^c)^\dagger=0_{\bbz_2}\right\}.
\end{equation*}
Let ${\mathcal X} = \{ x^i \}$ be a solution to \eqref{A-2} on $\bbr(\bbz_2)$. Then, we claim that ${\mathcal X}$ is an equilibrium for \eqref{A-2}: 
\begin{equation} \label{C-3-4}
x^c(x^i)^\dagger-x^i(x^c)^\dagger=0_{\bbz_2} \quad \mbox{i.e.,} \quad \frac{dx^i}{dt} = 0, \quad t > 0, \quad \forall~i=1, 2, \cdots, N.
\end{equation}
{\it Proof of \eqref{C-3-4}}: Note that 
\begin{equation} \label{C-3-5}
g^2 = e \quad \forall~g \in \bbz_2, \quad g_1 g_2 = g_2 g_1, \quad g_i \in \bbz_2.
\end{equation}
Then, we have
\begin{align*}
x^c(x^i)^\dagger-x^i(x^c)^\dagger&=\sum_{g_1, g_2}\left(x^c_{g_1}\overline{x^i_{g_2}}g_1g_2^{-1}-x^i_{g_1}\overline{x^c_{g_2}}g_1g_2^{-1}\right) =\sum_{g_1, g_2}\left(x^c_{g_1}{x}^i_{g_2}g_1g_2 -x^i_{g_2}{x}^c_{g_1}g_2g_1 \right) \\
&=\sum_{g_1, g_2}\left(x^c_{g_1}{x}^i_{g_2}-x^i_{g_2}{x}^c_{g_1}\right)g_1g_2=0_{\bbz_2}.
\end{align*}
Thus, $X$ is an equilibrium.
\end{proof}
\begin{remark} 
Note that the same argument employed in the proof can be applied to the following group: 
\[
\bigoplus_{i=1}^n\bbz_2 \simeq\{\underbrace{(\pm1,\pm1,\cdots,\pm1)}_{n\text{ times}}\}, \quad \forall n\in \mathbb{N}
\]
so that 
\[ {\mathcal E}\left(\bbr\left[\bigoplus_{i=1}^n\bbz_2\right] \right) = \bbr\left[  \bigoplus_{i=1}^n\bbz_2   \right]^N. \]
\end{remark}

\subsection{Geometric structure of ${\mathcal E}(\bbr[\bbz_3])$} \label{sec:3.3} Consider the case $\bbz_3 = \{e, a, a^2 \}$. Then, any element $x \in \bbr[\bbz_3]$ with a unit norm has the unique representation for $x$:
\[ x = x_e e + x_a a + x_{a^2} a^2, \quad x_e^2 + x_a^2 + x_{a^2}^2 = 1. \]
Note that $\bbr[\bbz_3]$ can be embedded into $\bbs^2$: there exists an injective map $\phi$:
\[
\phi: \bbr[\bbz_3]\to  \bbs^2 \subset \bbr^3,\quad \mbox{by}~~\phi(x)=(x_e, x_a, x_{a^2}).
\]
\begin{proposition}\label{P3.2}
The equilibrium manifold ${\mathcal E}(\bbr[\bbz_3])$ can be a disjoint union of the following three submanifolds:
\[
 {\mathcal E}(\bbr[\bbz_3]) = {\mathcal E}_1(\bbr[\bbz_3]) \cup {\mathcal E}_2(\bbr[\bbz_3]) \cup {\mathcal E}_3(\bbr[\bbz_3]), 
\]
where 
\begin{align}
\begin{aligned} \label{C-13}
{\mathcal E}_1(\bbr[\bbz_3]) &:= \{\{x^i\}_{i=1}^N \subset \bbr[\bbz_3]:~ x^c=0_{\bbz_3}\}, \\
{\mathcal E}_2(\bbr[\bbz_3]) &:= \{\{x^i\}_{i=1}^N \subset \bbr[\bbz_3]:~\langle x^i, y\rangle=0\mbox{ for some } y\neq 0_{\bbz_3} \}, \\
{\mathcal E}_3(\bbr[\bbz_3]) &:= \{\{x^i\}_{i=1}^N \subset \bbr[\bbz_3]:~\phi(x^c)  \sslash (1, 1, 1)\}.
\end{aligned}
\end{align}
\end{proposition}
\begin{proof}
Suppose that $\{x^i\}_{i=1}^N$ is an equilibrium solution:
\begin{equation} \label{C-4}
x^c(x^i)^\dagger = x^i(x^c)^\dagger, \quad i = 1, \cdots, N. 
\end{equation}
Now we set 
\begin{equation} \label{C-5}
 x^c = x_e^c e + x_a^c a + x_{a^2}^c a^2 \quad \mbox{and} \quad  x^i = x_e^i e + x_a^i a + x_{a^2}^i a^2.
\end{equation} 
Then, their hermitian conjugates are
\begin{equation} \label{C-6}
(x^i)^{\dagger} = x_e^i  e + x_{a^2}^i  a + x_a^i  a^2 \quad \mbox{and} \quad (x^c)^{\dagger} = x_e^c e + x_{a^2}^c a +  x_a^c a^2,
\end{equation}
where we used 
\[ e^{-1} = e, \quad a^{-1} = a^2, \quad (a^2)^{-1} = a. \]
In \eqref{C-4}, we use \eqref{C-5} and \eqref{C-6} to find that for each $i = 1, \cdots, N$, 
\begin{align}
\begin{aligned} \label{C-7}
x^c(x^i)^\dagger &=  x_e^c  x_e^i   e*e + x_e^c  x_{a^2}^i e* a + x_e^c x_a^i  e* a^2  + x_a^c x_e^i  a * e + x_a^c  x_{a^2}^i a* a +  x_a^c  x_a^i  a* a^2 \\
&\hspace{0.2cm} + x_{a^2}^c  x_e^i  a^2 * e +  x_{a^2}^c x_{a^2}^i  a^2 * a + x_{a^2}^c  x_a^i a^2 * a^2 \\
&= \Big(   x_e^c  x_e^i  +  x_a^c  x_a^i  +  x_{a^2}^c x_{a^2}^i  \Big) e + \Big( x_e^c  x_{a^2}^i + x_a^c x_e^i  +  x_{a^2}^c  x_a^i  \Big) a \\
&\hspace{0.2cm}+ \Big(  x_e^c x_a^i  +  x_a^c  x_{a^2}^i  +  x_{a^2}^c  x_e^i          \Big) a^2.
\end{aligned}
\end{align}
Similarly, one has 
\begin{align}
\begin{aligned} \label{C-8}
x^i (x^c)^{\dagger} &= \Big( x_e^i x_e^c + x_a^i x_a^c + x_{a^2}^i x_{a^2}^c     \Big) e + \Big(  x_e^i  x_{a^2}^c + x_a^i x_e^c  +  x_{a^2}^i  x_a^c \Big) a \\
&\hspace{0.2cm}+ \Big(  x_e^i  x_a^c +  x_a^i  x_{a^2}^c  + x_{a^2}^i x_e^c \Big) a^2.
\end{aligned}
\end{align}
We compare \eqref{C-7} and \eqref{C-8} to get 
\begin{align}
\begin{aligned}\label{C-9}
e:\quad  x_e^c  x_e^i  +  x_a^c  x_a^i  +  x_{a^2}^c x_{a^2}^i  = x_e^i x_e^c + x_a^i x_a^c + x_{a^2}^i x_{a^2}^c, \\
a:\quad x_e^c  x_{a^2}^i + x_a^c x_e^i  +  x_{a^2}^c  x_a^i  =  x_e^i  x_{a^2}^c + x_a^i x_e^c  +  x_{a^2}^i  x_a^c, \\
a^2:\quad x_e^c x_a^i  +  x_a^c  x_{a^2}^i  +  x_{a^2}^c  x_e^i  = x_e^i  x_a^c +  x_a^i  x_{a^2}^c  + x_{a^2}^i x_e^c.
\end{aligned}
\end{align}
Note that the relation $\eqref{C-9}_1$ trivially true, and if we set 
\[
\phi^i :=\phi(x^i), \quad i.e., \quad  (x_e^i, x_a^i, x_{a^2}^i) = (\phi_1^i, \phi_2^i, \phi_3^i) \in\bbr^3.
\]
Then relations $\eqref{C-9}_2$ and $\eqref{C-9}_3$ reduce to the following single relation:
\[
\phi^c_1\phi^i_3+\phi^c_2\phi^i_1+\phi^c_3\phi^i_2=\phi^i_1\phi^c_3+\phi^i_2\phi^c_1+\phi^i_3\phi^c_2,
\]
which can be rewritten as 
\begin{equation} \label{C-10}
(\phi_2^c\phi_3^i- \phi_3^c\phi^i_2) +(\phi_3^c\phi_1^i- \phi_1^c\phi^i_3) + ( \phi_1^c\phi_2^i - \phi_2^c\phi^i_1) =0.
\end{equation}
We can further simplify the relation \eqref{C-10} as 
\[
(1, 1, 1)\cdot (\phi^c\times \phi^i)=0.
\]
By the vector identity, this yields
\begin{equation} \label{C-11}
\phi^i\cdot((1, 1, 1)\times \phi^c)=0.
\end{equation}
The analysis for the relation \eqref{C-11} yields the complete classification of ${\mathcal E}(\bbr[\bbz_3])$. Now, we consider all cases satisfying the relation \eqref{C-11}. \newline

\noindent $\bullet$~Case A: Suppose 
\[ (1, 1, 1)\times \phi^c = 0. \]
In this case, we have the following two cases:
\[ \mbox{Either}~\phi^c = 0 \quad \mbox{or} \quad 0 \not = \phi^c  \sslash (1,1,1). \]

\vspace{0.2cm}

\noindent $\bullet$~Case B: Suppose
\[ (1, 1, 1)\times \phi^c \not = 0. \]
In this case, if we set $y := (1, 1, 1)\times \phi^c \not = 0$, then
\[  \langle \phi^i, y \rangle = \langle x^i, y \rangle = 0, \quad \mbox{for all $i = 1, \cdots, N$}. \]
Hence, it follows from Case A and Case B that the equilibrium set ${\mathcal E}(\bbr[\bbz_3])$ in \eqref{C-1} can be decomposed into three pieces:
\begin{equation*} \label{C-12}
 {\mathcal E}(\bbr[\bbz_3]) = {\mathcal E}_1(\bbr[\bbz_3]) \cup {\mathcal E}_2(\bbr[\bbz_3]) \cup {\mathcal E}_3(\bbr[\bbz_3]), 
\end{equation*}
where 
\begin{align*}
\begin{aligned} \label{C-13}
{\mathcal E}_1(\bbr[\bbz_3]) &:= \{\{x^i\}_{i=1}^N \subset \bbr[\bbz_3]:~x^c=0_G\}, \\
{\mathcal E}_2(\bbr[\bbz_3]) &:= \{\{x^i\}_{i=1}^N\subset \bbr[\bbz_3]:~\langle x^i, y\rangle=0\mbox{ for some } y\neq 0_G\}, \\
{\mathcal E}_3(\bbr[\bbz_3]) &:= \{\{x^i\}_{i=1}^N\subset \bbr[\bbz_3]:~\phi(x^c)  \sslash (1, 1, 1)\}.
\end{aligned}
\end{align*}
If we visualize the element of $\mathcal{E}_2(\bbr[\bbz_3])$ with embedding $\phi$, $\{\phi(x^i)\}_{i=1}^N$ lies on the same great circle as following figure.
\begin{figure}[h]\label{F1}
\includegraphics[width=7cm]{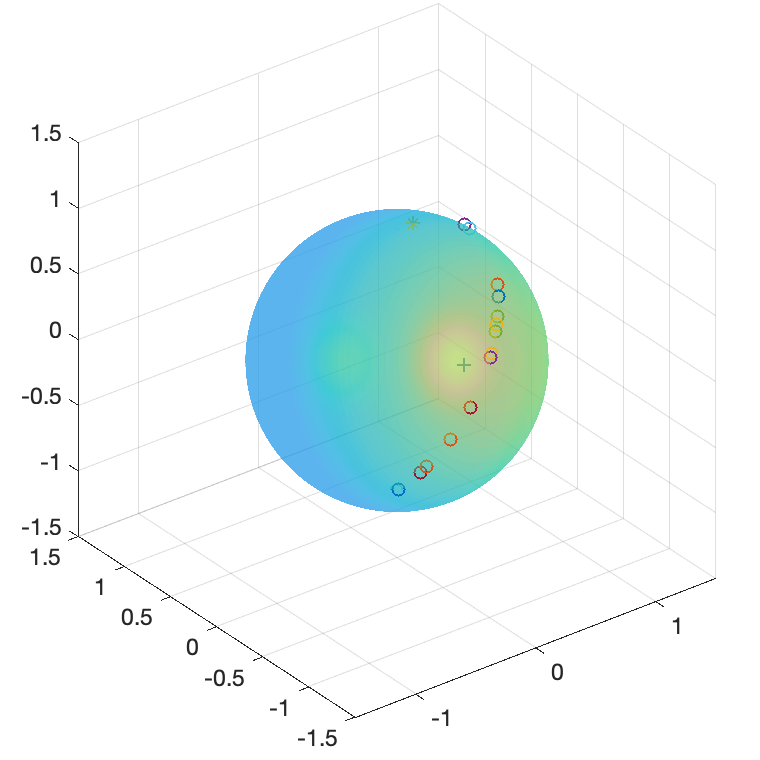}\quad
\includegraphics[width=7cm]{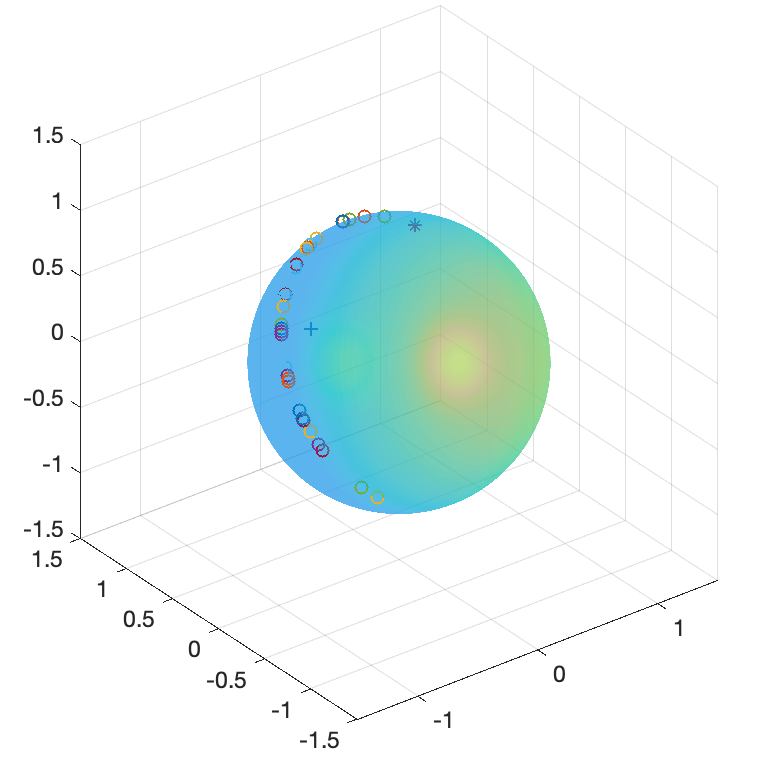}
\caption{Geometric diagrams of $\mathcal{E}_2(\bbr[\bbz_3])$}
\end{figure}
\end{proof}
In the following section, we consider the structure of the equilibrium manifold ${\mathcal E}(\bbr[G])$ for a group ring $\bbr[G]$. 
\section{Geometry of the equilibrium manifold} \label{sec:4}
\setcounter{equation}{0}
In this section, we extend the results for $\bbr[\bbz_3]$ to the equilibrium manifold ${\mathcal E}(\bbr[G])$ for $\bbr[G]$. \newline

Recall the equilibrium set for \eqref{A-2}:
\begin{equation} \label{D-3-0}
\mathcal{E}(\bbr[G])=\left\{ Y  \in \bbr[G]^N:y^c(y^i)^\dagger-y^i(y^c)^\dagger=0_G, \quad i  = 1, \cdots, N~~\right\},
\end{equation}
where $y^c := \frac{1}{N} \sum_{i=1}^{N} y^i$. \newline

\subsection{Geometric structure of an equilibrium manifold} \label{sec:4.1}
To provide an alternative representation of ${\mathcal E}(\bbr[G])$, we introduce an elementary matrix-like object $A^g \in \bbr^{G \times G}$ for each $g \in G$:
\begin{align}\label{D-3}
[A^g]_{g_1,g_2} :=\begin{cases}
1,\quad\text{when } g_1g_2^{-1}=g,\\
0,\quad\text{otherwise}.
\end{cases}
\end{align}
In the following lemma, we study several properties of $A^g$. 
\begin{lemma}\label{L4.1}
For $g \in G$, let $A^g$ be a matrix-like object defined by \eqref{D-3}. Then, it satisfies
\begin{align}
\begin{aligned}\label{D-3-0-0}
 &[A^g]_{g_1,g_2}=[A^{g^{-1}}]_{g_2,g_1}, \quad (A^g)^T=A^{g^{-1}}, \quad A^gA^{\tilde{g}}=A^{g\tilde{g}},\\
 & A^{g_1}\neq A^{g_2} \quad \mbox{if $g_1 \neq g_2$}.
 \end{aligned}
 \end{align}
\end{lemma}
\begin{proof}
\noindent $\bullet$~(First identity):~For $g \in G$, we use \eqref{D-3} to see
\[
[A^g]_{g_1,g_2}=1 \quad \Longleftrightarrow \quad g_1g_2^{-1}=g \quad \Longleftrightarrow \quad g_2g_1^{-1}=g^{-1} \quad \Longleftrightarrow \quad [A^{g^{-1}}]_{g_2,g_1}=1.
\]

\vspace{0.2cm}

\noindent $\bullet$~(Second identity):~Now we take a transpose the first relation: for $g_1, g_2 \in G$,  
\[ [(A^g)^T]_{g_1,g_2} = [A^g]_{g_2, g_1} = [A^{g^{-1}}]_{g_1, g_2}, \]
which yields the desired second identity.\newline

\vspace{0.2cm}

\noindent $\bullet$~(Third identity):~It suffices to check
\[ \sum_{g_2\in G}[A^g]_{g_1,g_2}[A^{\tilde{g}}]_{g_2,g_3}=[A^{g\tilde{g}}]_{g_1,g_3}.  \]
Recall Kronecker-delta function defined on a group $G$:
\[
\delta_{g_1, g_2}=\begin{cases}
1,\quad \text{when } g_1=g_2 \in G,\\
0,\quad \text{otherwise}.
\end{cases}
\]
For $g_1, g_3 \in G$, one has 
\begin{equation} \label{D-3-1}
\sum_{g_2\in G}[A^g]_{g_1,g_2}[A^{\tilde{g}}]_{g_2,g_3}=\sum_{g_2\in G}\delta_{g, g_1g_2^{-1}}\delta_{\tilde{g}, g_2g_3^{-1}}.
\end{equation}
On the other hand, we use 
\[
\delta_{g, g_1g_2^{-1}}=\delta_{gg_2, g_1}=\delta_{g_2, g^{-1}g_1}
\]
to find 
\[
g_2=g^{-1}g_1\quad\Longleftrightarrow \quad \delta_{g, g_1g_2^{-1}}=1.
\]
Thus, one has
\begin{equation} \label{D-3-2}
\sum_{g_2\in G}\delta_{g, g_1g_2^{-1}}\delta_{\tilde{g}, g_2g_3^{-1}}=\delta_{\tilde{g}, g_2g_3^{-1}}\Big|_{g_2=g^{-1}g_1}=\delta_{\tilde{g}, g^{-1}g_1g_3^{-1}}=\delta_{g\tilde{g}, g_1g_3^{-1}}=[A^{g\tilde{g}}]_{g_1,g_3}.
\end{equation}
Finally, we combine \eqref{D-3-1} and \eqref{D-3-2} to find the desired estimate. 

\vspace{0.2cm}

\noindent $\bullet$~(Forth identity): Suppose $g_1\neq g_2$. Then, we have
\[
[A^{g_1}]_{g_1, e}=1,\quad [A^{g_2}]_{g_2, e}=0,
\]
which yields $A^{g_1}\neq A^{g_2}$.
\end{proof}
\begin{remark} \label{R4.1}
Since $A^g$ is a matrix-like object with real entries, the hermitian conjugate of $A^g$ is simply the transpose of it:
\[
[(A^g)^\dagger]_{g_1g_2}=[A^g]_{g_2g_1}.
\]
Note that elements of $A^g$ are real, so the transpose of $A^g$ can be defined as above.
\end{remark}
Next, we provide our third set of results in the following theorem. 
\begin{theorem} \label{T4.1}
The following assertions hold.
\begin{enumerate}
\item
The equilibrium manifold \eqref{D-3-0} can be rewritten as 
\[
\mathcal{E}(\bbr[G])=\mathcal{E}_0(\bbr[G]) \sqcup \left(\bigsqcup_{a\in\mathcal{I}}\left(\bigcap_{g\in G}\mathcal{E}^g_{a_g}(\bbr[G]) \right)\right),
\]
where the index set $\mathcal{I}=\{a\in\bbc[G]: a_g=1\text{ or }2\quad\forall g\in G\}$, and 
\begin{align}
\begin{aligned}\label{D-3-2-1}
\mathcal{E}_0(\bbr[G])&:=\left\{ Y \in\bbr[G]^N: y^c=0_G\right\},\\
\mathcal{E}_1^g(\bbr[G])&:=\left\{ Y \in\bbr[G]^N: y^c\left(A^g-A^{g^{-1}}\right)=0_G\right\}\backslash\mathcal{E}_0(\bbr[G]) \\
&=\left\{ Y \in\bbr[G]^N: y^c\in\mathrm{Null}\left(A^g-A^{g^{-1}}\right)\backslash \{0_G\}\right\},\\
\mathcal{E}_2^g(\bbr[G])&:=\left\{ Y \in\bbr[G]^N: y^c\left(A^g-A^{g^{-1}}\right)(y^i)^T=0_G\right\}\backslash (\mathcal{E}_0(\bbr[G]) \cup\mathcal{E}_1^g(\bbr[G])).
\end{aligned}
\end{align}
\item
For $g \in G$, the null space $\mathrm{Null}\left(A^g-A^{g^{-1}}\right)$ appearing in the defining relations of $\mathcal{E}_1^g(\bbr[G])$ and $\mathcal{E}_2^g(\bbr[G])$ in \eqref{D-3-2-1} satisfies 
\begin{align}
\begin{aligned}  \label{New-1}
& \mathrm{Null}\left(A^g-A^{g^{-1}}\right)= \Big \{x\in\bbr[G]: x_{g_1}=x_{g_2}\quad\text{if }~H(g) g_1= H(g) g_2 \Big \}, \\
& \mathrm{Nullity}\left(A^g-A^{g^{-1}}\right)=\frac{|G|}{|H(g)|},
\end{aligned}
\end{align}
where $H(g)$ is a subgroup of $G$ defined by 
\[H(g) := \langle g^2\rangle=\{g^{2n}:n\in\bbz\}. \]
\end{enumerate}
\end{theorem}
\begin{proof} Since the proof is rather lengthy, we leave its proof in the following two subsections.
\end{proof}

\subsection{Proof of the first assertion} \label{sec:4.2} In this subsection, we verify the proof of the first part of Theorem \ref{T4.1} in two steps. \newline

\noindent $\bullet$~(Step A):~For a fixed $g\in G$, we define the manifold $\mathcal{E}^g(\bbr[G])$ as 
\begin{equation} \label{D-3-2-1-1}
\mathcal{E}^g(\bbr[G]) := \Big \{ Y \in\bbr[G]^N:~y^c A^g(y^i)^T=y^i A^g(y^c)^T,\quad\forall~i = 1, \cdots, N \Big \}.
\end{equation}
Then, we claim:
\begin{equation} \label{D-3-2-2}
\mathcal{E}(\bbr[G])=\bigcap_{g\in G} \mathcal{E}^g(\bbr[G]),
\end{equation}
{\it Proof of \eqref{D-3-2-2}}:~By definition of $A^g$ in \eqref{D-3}, one has 
\[
y^c(y^i)^\dagger =\sum_{g_1, g_2\in G} y^c_{g_1}\overline{y^i_{g_2}} g_1g_2^{-1}=\sum_{g\in G}\left(\sum_{g_1g_2^{-1}=g}  y^c_{g_1}\overline{y^i_{g_2}}\right)g =\sum_{g\in G}\left(\sum_{g_1, g_2\in G} y^c_{g_1}[A^g]_{g_1g_2} \overline{y^i_{g_2}}\right)g.
\]
This yields
\[
\left(y^c(y^i)^\dagger\right)_g=\sum_{g_1, g_2\in G} y^c_{g_1}[A^g]_{g_1,g_2} \overline{y^i_{g_2}}.
\]
By treating $y^i$ and $y^c$ as row vectors, one has
\begin{equation} \label{D-3-3}
y^c A^g(y^i)^\dagger=\left(y^c(y^i)^\dagger\right)_g.
\end{equation}
Similarly, we have
\begin{equation} \label{D-3-4}
y^i A^g(y^c)^\dagger=\left(y^i(y^c)^\dagger\right)_g.
\end{equation}
Therefore, the relations \eqref{D-3-3} and \eqref{D-3-4} yield
\[
y^c(y^i)^\dagger=y^i(y^c)^\dagger\quad\Longleftrightarrow\quad y^c A^g(y^i)^\dagger=y^i A^g(y^c)^\dagger,\quad \forall~g\in G.
\]
Finally, we verify \eqref{D-3-2-2}:
\begin{align}\label{D-4}
\begin{aligned}
\mathcal{E}(\bbr[G]) &=\left\{\{y^i\}\in\bbr[G]^N: y^c A^g(y^i)^\dagger=y^i A^g(y^c)^\dagger,\quad\forall g\in G,~\forall  i = 1, \cdots, N \right\}  \\
&= \bigcap_{g \in G} \left\{\{y^i\}\in\bbr[G]^N: y^c A^g(y^i)^\dagger=y^i A^g(y^c)^\dagger,\quad \forall  i = 1, \cdots, N \right\} \\
& = \displaystyle \bigcap_{g \in G} {\mathcal E}^g(\bbr[G]).
\end{aligned}
\end{align}
Before we proceed further, we give a comment on the representation \eqref{D-4} for system \eqref{A-2} on $\bbr[\bbz_2]$. \newline

Consider system \eqref{A-2} defined on $\bbr[\bbz_2]$. In this case, due to \eqref{C-3-5}, one has
\[ A^g=A^{g^{-1}} \quad \mbox{and} \quad  x^i\in\bbr[\bbz_2] \quad \mbox{for all $ i = 1, \cdots, N$}. \]
Then,  we use the above relations and $x^cA^g(x^i)^T \in \bbr$ to get 
\[
x^cA^g(x^i)^\dagger= x^cA^g(x^i)^T=\left(x^cA^g(x^i)^T\right)^T=x^i \left(A^g\right)^T (x^c)^T = x^i A^{g^{-1}} (x^c)^T =x^iA^g(x^c)^\dagger.
\]
Hence $X \in {\mathcal E}(\bbr[G])$. \newline

\noindent $\bullet$~(Step B):~We further analyze the set ${\mathcal E}^g(\bbr[G])$ as follows.  Suppose that $Y \in {\mathcal E}(\bbr[G])$. Then, one has 
\begin{equation} \label{D-5}
y^c A^g(y^i)^T=y^i A^g(y^c)^T \in \bbr \quad\forall g\in G,\quad\forall i = 1, \cdots, N.
\end{equation}
Since each component is real-valued, we can use a transpose instead of $\dagger$. Then, we use  the conditions \eqref{D-5}  to see
\[
y^c A^g(y^i)^T=y^i A^g(y^c)^T=\left(y^i A^g(y^c)^T\right)^T=y^c (A^{g})^T (y^i)^T.
\]
This yields
\[
y^c\left(A^g-(A^g)^T\right)(y^i)^T=0
\]
for all $g\in G$ and $ i = 1, \cdots, N$. \newline

\noindent It follows from the second identity in Lemma \ref{L4.1} that 
\[
y^c\left(A^g-(A^g)^T\right)(y^i)^T=0\quad \Longrightarrow\quad y^c\left(A^g-A^{g^{-1}}\right)(y^i)^T=0.
\]
Thus, the set ${\mathcal E}^g(\bbr[G])$ in \eqref{D-3-2-1-1} can be rewritten as 
\begin{equation} \label{D-5-0}
\mathcal{E}^g(\bbr[G]) := \Big\{ Y \in\bbr[G]^N:~ y^c\left(A^g-A^{g^{-1}}\right)(y^i)^T=0,\quad\forall~i = 1, \cdots, N  \Big \}.
\end{equation}
Now we can decompose $\mathcal{E}^g(\bbr[G])$ as a disjoint union of three sets:
\begin{equation} \label{D-5-0-0}
\mathcal{E}^g(\bbr[G])=\mathcal{E}_0(\bbr[G])\sqcup \mathcal{E}_1^g(\bbr[G])\sqcup\mathcal{E}_2^g(\bbr[G]).
\end{equation}
Finally, we combine \eqref{D-3-2-2} and \eqref{D-5-0-0} to see 
\[
\mathcal{E}(\bbr[G])=\mathcal{E}_0(\bbr[G])\sqcup \left(\bigsqcup_{a\in\mathcal{I}}\left(\bigcap_{g\in G}\mathcal{E}^g_{a_g}(\bbr[G])\right)\right),
\]
where the index set $\mathcal{I}=\{a\in\bbr(G): a_g=1\text{ or }2\quad\forall g\in G\}$. \newline

\subsection{Proof of the second assertion} \label{sec:4.3}
Our second statement in Theorem \ref{T4.1} is concerned with the structure of $\mathrm{Null}\left(A^g-A^{g^{-1}}\right)$ appearing in the set $\mathcal{E}_1^g(\bbr[G])$ of \eqref{D-3-2-1}.\newline

 Recall the subgroup $H(g)$:
\[H(g):= \langle g^2\rangle=\{g^{2n}:n\in\bbz\}, \]
and we claim: for $g \in G$, 
\begin{align*}
\begin{aligned}
& (i)~\mathrm{Null}\left(A^g-A^{g^{-1}}\right)= \Big \{x\in\bbr(G): x_{g_1}=x_{g_2}\quad\text{for $g_1$ and $g_2$ such that}~H(g) g_1= H(g) g_2 \Big \}. \\
& (ii)~\mathrm{Nullity}\left(A^g-A^{g^{-1}}\right)=\frac{|G|}{|H(g)|}.
\end{aligned}
\end{align*}
\noindent $\bullet$~(Derivation of (i)):~Let $g \in G$ be any, and suppose that $x \in \mathrm{Null}\left(A^g-A^{g^{-1}}\right)$:
\[
\left(A^g-A^{g^{-1}}\right)x=0_G.
\]
Then, $g_1$-th component of above relation is
\begin{equation} \label{D-5-2}
\sum_{g_2\in G}\left[A^g-A^{g^{-1}}\right]_{g_1, g_2}~x_{g_2}=0.
\end{equation}
By definition of $A^g$ in \eqref{D-3}, one has 
\begin{align}
\begin{aligned} \label{D-5-3}
\sum_{g_2\in G}\left[A^g-A^{g^{-1}}\right]_{g_1, g_2}~x_{g_2}&=\sum_{g_2\in G}\left(\delta_{g, g_1g_2^{-1}}-\delta_{g^{-1}, g_1g_2^{-1}}\right)x_{g_2}\\
& =\sum_{g_2\in G}\left(\delta_{g_2, g^{-1}g_1}-\delta_{g_2, gg_1}\right)x_{g_2} =x_{g^{-1}g_1}-x_{gg_1}.
\end{aligned}
\end{align}
It follows from \eqref{D-5-2} and \eqref{D-5-3} that
\[
x_{g^{-1}g_1}=x_{gg_1}\quad\forall g_1\in G.
\]
This yields
\begin{align}\label{D-6}
x_{g_1}=x_{g^2g_1}\quad\forall  g_1\in G.
\end{align}
Hence, we have shown that 
\begin{align}
\begin{aligned} \label{D-6-1}
x \in \mathrm{Null}\left(A^g-A^{g^{-1}}\right) &\quad \Longleftrightarrow \quad x_{g_1}=x_{g^2g_1}\quad\forall  g_1\in G \\
&\quad \Longleftrightarrow \quad x_{g_1}=x_{g^{2k}g_1}\quad\forall  g_1\in G,~k \in \bbz.
\end{aligned}
\end{align} 
On the other hand, note that 
\begin{equation} \label{D-6-2}
H(g) g_1 = H(g) g_2 \quad \Longleftrightarrow\quad g_2 = (g^2)^k g_1 \quad \mbox{for some $k \in \bbz$}. 
\end{equation}
Then, it follows from \eqref{D-6-1} and \eqref{D-6-2} that 
\begin{equation} \label{D-6-3}
x \in \mathrm{Null}\left(A^g-A^{g^{-1}}\right) \quad \Longleftrightarrow \quad x_{g_1} = x_{g_2} \quad \mbox{for $g_1, g_2$ such that $H(g) g_1 = H(g) g_2$}.
\end{equation}
which verifies the first estimate (i). \newline

\noindent $\bullet$~(Derivation of the estimate (ii)): By the relation \eqref{D-6-3}, if $x \in \mathrm{Null}\left(A^g-A^{g^{-1}}\right)$, then the coefficients in the same right coset of $H(g)$ must be the same. Hence, the dimension $M$ of $\mathrm{Null}\left(A^g-A^{g^{-1}}\right)$ must be the same with the number of distinct right cosets of $H(g)$, i.e.,  we can choose a subset $\{\tilde{g}_1,\cdots, \tilde{g}_M\}$ of $G$ with following property:
\[
H\tilde{g}_i\neq H\tilde{g}_j
\]
for all $1\leq i\neq j\leq M$.  Since each right coset has the same cardinality (by Theorem of Lagrange), one has 
\[
M = \frac{|G|}{|H(g)|}.
\]
This verifies the second estimate (ii).

\vspace{0.5cm}

\subsection{Application to the cyclic group} \label{sec:4.4}
In this subsection, we further analyze the detailed structure of an equilibrium set for the cyclic group $G=\bbz_n$.  As a direct corollary of the argument in the proof of the second statement of Theorem \ref{T4.1}, one has the following assertions with the cyclic group $G=\bbz_n=\{\bar{0}, \bar{1}, \cdots, \overline{n-1}\}$.
\begin{corollary} \label{C4.1}
Let $\bbz_n$ be a cyclic group with order $n$.(i.e. $G= \bbz_n$). Then, the following assertions hold.
\begin{enumerate}
\item
If $\bar{m}\in \bbz_n$, then we have
\[
\mathrm{Nullity}\left(A^{\bar{m}}-A^{\bar{m}^{-1}}\right)=\frac{n}{|\langle \overline{2m}\rangle|}=\mathrm{gcd}(2m, n).
\]
\item
If $n = p \geq 3$ and $p$ is a prime, then we have
\[
\mathrm{Nullity}\left(A^g-A^{g^{-1}}\right)=\begin{cases}
p,\quad \text{when } g=\bar{0},\\
1,\quad \text{otherwise}.
\end{cases}
\]
\end{enumerate}
\end{corollary}
For a later use, we also present the following useful lemma.
\begin{lemma}\label{L4.2}
Let $G$ be a finite group. Then, for $g \in G$ and $n \in \bbz$, one has
\begin{equation} \label{D-7}
\mathrm{Null}\left(A^g-A^{g^{-1}}\right)\subseteq \mathrm{Null}\left(A^{g^n}-A^{g^{-n}}\right).
\end{equation}
\end{lemma}
\begin{proof}
By symmetry in \eqref{D-7}, it suffices to consider the case when $n\geq0$. \newline

\noindent $\bullet$~Case A ($n=0$): In this case, the estimate \eqref{D-7} is trivially true since 
\[  A^{g^0}-A^{g^0}=O \quad \mbox{and} \quad \mathrm{Null}\left(A^{g^0}-A^{g^{0}}\right) = \bbr[G]. \]

\vspace{0.2cm}

\noindent $\bullet$~Case B ($n=1$):  Again, this relation is tautology.   \newline

\noindent $\bullet$~Case C ($n \geq 2$): it follows from $\eqref{D-3-0-0}_3$ in Lemma \ref{L4.1} (ii) that 
\[
\left(A^g-A^{g^{-1}}\right)\left(A^{g^{n-1}}+\cdots+A^{g}+A^{g^0}+A^{g^{-1}}+\cdots+A^{g^{-(n-1)}}\right)=A^{g^n}-A^{g^{-n}}.
\]
If $x\left(A^g-A^{g^{-1}}\right)=0_G$, 
\begin{align*}
\begin{aligned}
x \Big(A^{g^n}-A^{g^{-n}} \Big) &= x\left(A^g-A^{g^{-1}}\right)\left(A^{g^{n-1}}+\cdots+A^{g}+A^{g^0}+A^{g^{-1}}+\cdots+A^{g^{-(n-1)}}\right) \\
&= 0_G. 
\end{aligned}
\end{align*}
This implies the desired relation.
\end{proof}
Next, we further investigate the structure of the equilibrium manifold for the cyclic group $\bbz_p$ with a prime order $p\geq3$. It follows from Lemma \ref{L4.2} that one has the following corollary. 
\begin{proposition} \label{P4.1}
Let $\bbz_p$ be the cyclic group with prime order $p\geq3$. Then, for all $g, \tilde{g}\in\bbz_p\backslash\{e\}$,
\[
\mathrm{Null}\left(A^g-A^{g^{-1}}\right)=\mathrm{Null}\left(A^{\tilde{g}}-A^{\tilde{g}^{-1}}\right).
\]
\end{proposition}
\begin{proof}
Let $g, \tilde{g}\in\bbz_p\backslash\{e\}$. By the property of $\bbz_p$, $g$ and $\tilde{g}$ are generators of $\bbz_p$. This means there exist two integer $n, \tilde{n}$ such that 
\[
g^n=\tilde{g},\quad \tilde{g}^{\tilde{n}}=g.
\]
Now we apply Lemma \ref{L4.2} to obtain
\[
\mathrm{Null}\left(A^g-A^{g^{-1}}\right)\subseteq \mathrm{Null}\left(A^{\tilde{g}}-A^{\tilde{g}^{-1}}\right),\quad \mathrm{Null}\left(A^{\tilde{g}}-A^{\tilde{g}^{-1}}\right)\subseteq \mathrm{Null}\left(A^{g}-A^{g^{-1}}\right).
\]
This yields the desired estimate.
\end{proof}
\begin{remark}\label{R4.2}
By $\eqref{New-1}_2$, we can also find the explicit form of the null space of $A^g-A^{g^{-1}}$:
\[
\mathrm{Null}\left(A^g-A^{g^{-1}}\right)=\begin{cases}
\bbr(G),\quad\text{if } g=e,\\
\left\{x\in\bbr(G): x_{\tilde{g}}=a\in\bbr\quad\forall~ \tilde{g}\in G\right\},\quad\text{otherwise}.
\end{cases}
\]
Since $\mathrm{Null}\left(A^g-A^{g^{-1}}\right)$ does not depends on $g\in G$, $\mathcal{E}^g_1(\bbr[G])$ defined in \eqref{D-3-2-1} also does not depends on $g\in G$. This implies that, if $\{y^i\}\in\mathcal{E}_1^{g}(\bbr[G])$ for $g\neq e$, then it is easy to see that ${\mathcal Y} := \{y^i\}\in\mathcal{E}_1^{\tilde{g}}(\bbr[G])$ for all $\tilde{g}\in G\backslash\{e\}$.
\end{remark}

\vspace{0.5cm}

Based on \eqref{C-1} and Remark \ref{R4.2}, we define the following sets:
\begin{align*}
\begin{aligned}
\mathcal{L}&:=\left\{ Y \in\bbr[G]^N:~y^c\in\mathrm{Null}\left(A^g-A^{g^{-1}}\right)\quad\forall g\in G\right\},\\
\mathcal{C}&:=\bigcap_{g\in G}\left\{\{ Y \in\bbr[G]^N:~y^c\left(A^g-A^{g^{-1}}\right)(y^i)^T=0\right\}\backslash \mathcal{L}.
\end{aligned}
\end{align*}
\begin{proposition}\label{P4.2}
Let  $\{y^i \}\in\mathcal{C}$. Then, the following assertions hold.
\begin{enumerate}
\item
There exists $(|G|-1)$-dimensional hyperplane containing $\{y^i\}$.
\item
There exists a $(|G|-2)$-dimensional sphere centered at $0_G$ containing $\{y^i\}$. 
\end{enumerate}
\end{proposition}
\begin{proof}
Let $\{y^i\}\in\mathcal{C}$ and fix $g\in G\backslash \{e\}$. Since $\{y^i\}\not\in \mathcal{L}$, we know
\[
v(Y, g):=x^c\left(A^g-A^{g^{-1}}\right)\neq0_G.
\]
By definition of $\mathcal{C}$, we have
\[
v(Y, g) (y^i)^T=0, \quad \forall~i = 1, \cdots, N.
\]
This means that $y^i$ is orthogonal to non-zero vector $v(Y, g)$, and hence there exists  $(|G|-1)$-dimensional hyperplane which containing $\{y^i\}$ and origin $0_G$. Moreover, we can also see that there exists $(|G|-2)$-dimensional sphere centered at $0_G$ containing $\{y^i\}$(see Figure \ref{fig2}).
\end{proof}
\begin{center}
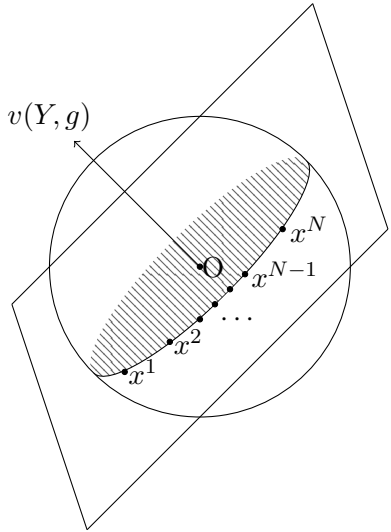
\begin{figure}[h]
\begin{tikzpicture}
\draw[black](0,0) circle [radius=2cm];
\greatcircle[black] {0,0} {2cm}{0.5cm}{45}
\draw[fill] (0,0) circle (1pt) node[xshift=5pt] {O}; 
\node (A) at (-2, 2) {$v(Y, g)$};
\draw [->] (0, 0) -- (A);
\draw [-] (-2+0.5, -2-1-0.5)-- (2+0.5,2-1-0.5);
\draw [-] (2+0.5,2-1-0.5) -- (2-0.5, 2+1+0.5);
\draw [-] (2-0.5, 2+1+0.5) -- (-2-0.5,-2+1+0.5 );
\draw [-] (-2-0.5,-2+1+0.5 ) -- (-2+0.5, -2-1-0.5);
\draw[fill] (-1,-1.4) circle (1pt) node[xshift=7pt] {$x^1$}; 
\draw[fill] (-0.4,-1.) circle (1pt) node[xshift=7pt] {$x^2$}; 
\draw[fill] (-0.,-0.7) circle (1pt) node[xshift=15pt] {$\cdots$}; 
\draw[fill] (0.2,-0.5) circle (1pt) node[xshift=10pt] {}; 
\draw[fill] (0.4,-0.3) circle (1pt) node[xshift=7pt] {}; 
\draw[fill] (0.6,-0.1) circle (1pt) node[xshift=15pt] {$x^{N-1}$}; 
\draw[fill] (1.1,0.5) circle (1pt) node[xshift=10pt] {$x^N$}; 
\end{tikzpicture}
\caption{Schematic diagram of the set $\mathcal{C}$}
\label{fig2}
\end{figure}
\end{center}

\section{Conclusion} \label{sec:5}
\setcounter{equation}{0}
In this paper, we have studied how the aggregation dynamics on a group ring can interplay with the group structure. For this, we have proposed a first-order aggregation model on group ring, and showed that the flow generated by the proposed model approaches to the equilibrium manifold for all initial data in a positive coupling regime. Our convergence result is based on the Lyapunov functional approach and the LaSalle invariance principle. Unfortunately, we cannot guarantee whether the flow will converge to some equilibrium along the full flow, since we do not have enough information on the explicit structure of the equilibrium manifold. In a full generic setting, we can decompose the equilibrium manifold as a disjoint union of three submanifolds. The geometric structure of one of such manifold depends on the group structure. For a cyclic group with prime order, we explicit analyzed the structures of the equilibrium sets. There are several issues that we do not consider here. For example, we did not show whether the flow can converge to some equilibrium along the full flow, i.e., we did not exclude a possibility of limit cycle in current analysis. Moreover, our convergence rate to the equilibrium manifold is also not explicit. Our proposed model is only for the aggregate modeling of identical particles. Thus, how to incorporate non-identical particles in our current modeling is also an interesting open question. We leave these interesting problems for a future work.

\end{document}